\documentclass{amsart}
\usepackage{graphicx}
\usepackage{amsmath}
\usepackage{amssymb}
\usepackage{oldgerm}
\usepackage[all]{xy}

\newcommand{\Hom}{\operatorname{Hom}\nolimits}
\renewcommand{\Im}{\operatorname{Im}\nolimits}

\newcommand{\Ker}{\operatorname{Ker}\nolimits}

\newcommand{\Tor}{\operatorname{Tor}\nolimits}
\newcommand{\Ext}{\operatorname{Ext}\nolimits}

\newcommand{\HH}{\operatorname{HH}\nolimits}
\newcommand{\Ho}{\operatorname{H}\nolimits}

\newcommand{\La}{\Lambda}

\newcommand{\op}{\operatorname{op}\nolimits}

\newcommand{\e}{\operatorname{e}\nolimits}
\newcommand{\Lae}{\Lambda^{\e}}
\newcommand{\Ae}{A^{\e}}

\newcommand{\Tatetor}{\operatorname{\widehat{Tor}}\nolimits}
\newcommand{\Tateext}{\operatorname{\widehat{Ext}}\nolimits}
\newcommand{\TateHH}{\operatorname{\widehat{HH}}\nolimits}

\newtheorem{theorem}{Theorem}[section]

\newtheorem{corollary}[theorem]{Corollary}

\newtheorem{lemma}[theorem]{Lemma}
\newtheorem{proposition}[theorem]{Proposition}
\newtheorem*{theoremempty}{Theorem}
\theoremstyle{definition}
\newtheorem*{definition}{Definition}
\theoremstyle{definition}

\theoremstyle{definition}

\theoremstyle{definition}

\theoremstyle{definition}

\theoremstyle{definition}

\theoremstyle{remark}
\newtheorem*{remark}{Remark}
\theoremstyle{definition}

\theoremstyle{definition}

\begin{document}
\title{Tate-Hochschild homology and cohomology of Frobenius algebras}
\author{Petter Andreas Bergh \& David A.\ Jorgensen}

\address{Petter Andreas Bergh \\ Institutt for matematiske fag \\
  NTNU \\ N-7491 Trondheim \\ Norway}
\email{bergh@math.ntnu.no}

\address{David A.\ Jorgensen \\ Department of mathematics \\ University
of Texas at Arlington \\ Arlington \\ TX 76019 \\ USA}
\email{djorgens@uta.edu}


\thanks{The second author was partly supported by NSA grant H98230-10-0197, and
performed partly during a visit by the second author to the first. The second
author thanks the Institutt for Matematiske Fag, NTNU, for their hospitality and
generous support.}

\keywords{Hochschild homology, Hochschild cohomology, Tate homology, Tate
cohomology, Frobenius algebras, complete resolutions}

\subjclass[2000]{Primary: 16E40, 16E30. Secondary: 16D50}

\begin{abstract}
Let $\Lambda$ be a two-sided Noetherian Gorenstein $k$-algebra, for $k$ a field.
We introduce Tate-Hochschild homology and cohomology groups for $\Lambda$, which are defined
for all degrees, non-negative as well as negative, and which agree with the usual Hochschild homology and cohomology groups for all degrees larger than the injective dimension of $\Lambda$. We prove certain
duality theorems relating the Tate-Hochschild (co)homology groups in positive degree to
those in negative degree, in the case where $\Lambda$ is a Frobenius algebra.  We explicitly compute all Tate-Hochschild (co)homology groups
for certain classes of Frobenius algebras, namely, certain quantum complete intersections.
\end{abstract}

\maketitle

\section{Introduction}\label{intro}
Hochschild cohomology was introduced by Hochschild in \cite{Hochschild1, Hochschild2} as a tool for
studying the structure of associative algebras.  A bit later, Tate introduced a cohomology theory based on complete resolutions, which consequently defined cohomology in all degrees, positive \emph{and} negative (cf. the end of \cite{Tate}). In this paper we combine these two notions of cohomology and extend Hochschild cohomology to the `negative side,'
arriving at what we call \emph{Tate-Hochschild cohomology}.
It turns out that the `positive side' of Tate-Hochschild cohomology agrees with the usual Hochschild cohomology. We show
that in some cases the `positive' and `negative' sides are symmetric. However, this is not the case in
general, and we illustrate this by computing explicitly both sides of Tate-Hochschild cohomology
for certain classes of algebras.

More specifically, let $k$ be a field and $\La$ denote a two-sided Noetherian
Gorenstein $k$-algebra.  Then $\La$ has a complete resolution $\mathbb T$ over the enveloping algebra
$\La^e$ of $\La$, and for a $\La$-$\La$-bimodule $B$ one can define the
Tate-Hochschild cohomology groups with coefficients in $B$ by
\[
\TateHH^n(\La,B)=\Ho^n(\Hom_{\La^e}(\mathbb T,B))
\]
for all $n\in\mathbb Z$. (See Section \ref{tatehochschild} for details.)

When $\La$ is a finite dimensional algebra and $B$ is finitely generated, then the Tate-Hochschild cohomology groups are finite dimensional vector spaces over $k$.  We prove in Section
\ref{tatehochschild}
general duality results which relate the vector space dimensions of the positive cohomology to
those of the negative cohomology with coefficients in a dual module.  We use these
results in Section \ref{frobeniusalg} to establish, for example, the following consequence when $\La$ is moreover a Frobenius algebra:

\begin{theoremempty} Let $\La$ be a Frobenius algebra, with Nakayama automorphism
$\nu$.  Then
\[
\dim_k \TateHH^n(\La,\La) = \dim_k\TateHH^{-(n+1)}(\La,{_{\nu^2}\La_1})
\]
for all $n\in\mathbb Z$, where ${_{\nu^2}\La_1}$ denotes the bimodule $\La$ twisted on the right by the automorphism $\nu^2$.
\end{theoremempty}

Thus Tate-Hochschild cohomology is symmetric when $\nu$ squares to the identity automorphism, and this
is the case, for example, when $\La$ is a symmetric algebra or an exterior algebra. On the other hand, for certain classes
of Frobenius algebras, Tate-Hochschild cohomology is not symmetric.  In Section \ref{quantumci} we compute
the Tate-Hochschild cohomology for the quantum complete intersection $A = k \langle X,Y \rangle / (X^a, XY-qYX, Y^b)$ with $a,b\ge 2$ and $q$ not a root of unity in $k$, finding that
\[
\dim \TateHH^n(A,A) = \left \{
\begin{array}{ll}
1 & \text{if } n=0 \\
2 & \text{if } n=1 \\
1 & \text{if } n=2 \\
0 & \text{if } n \neq 0,1,2. \\
\end{array}
\right.
\]

Throughout the paper we simultaneously treat the homology version as well, \emph{Tate-Hochschild homology}.
It turns out that the Tate-Hochschild homology behaves quite different than does the cohomology.
For example, in Section \ref{frobeniusalg} we give the companion to the theorem above, showing
that Tate-Hochschild homology is always symmetric when $\La$ is a Frobenius algebra. This result was first proved in \cite{EuSchedler}.

\begin{theoremempty} Let $\La$ be a Frobenius algebra.  Then
\[
\dim_k \TateHH_n(\La,\La) = \dim_k\TateHH_{-(n+1)}(\La,\La)
\]
for all $n\in\mathbb Z$.
\end{theoremempty}

Again, this theorem is a consequence of more general duality statements which we prove in Section \ref{tatehochschild}.

\section{Tate-Hochschild (co)homology}\label{tatehochschild}

Let $k$ be a commutative ring and $\La$ a $k$-algebra. We denote by $\La^{\op}$ the opposite algebra of $\La$, and by $\Lae$ the enveloping algebra $\La
\otimes_k \La^{\op}$ of $\La$. The $k$-dual $\Hom_k(-,k)$ is denoted by
$D(-)$, and the ring dual $\Hom_{\La}(-, \La )$ by $(-)^*$.

The classical Hochschild cohomology groups of an algebra were introduced by Hochschild in \cite{Hochschild1, Hochschild2}. For every non-negative integer $n$, let $Q_n$ denote the $n$-fold tensor product $\La \otimes_k \cdots \otimes_k \La$ of $\La$ over $k$, with $Q_0=k$. If $B$ is a $\La$-$\La$-bimodule, the corresponding \emph{Hochschild cohomology complex}
$$\cdots \to 0 \to 0 \to H^0 \xrightarrow{\partial^0} H^1 \xrightarrow{\partial^1} H^2 \xrightarrow{\partial^2} H^3 \to \cdots$$
is defined as follows:
$$H^n = \left \{
\begin{array}{ll}
0 & \text{for } n<0, \\
B & \text{for } n=0, \\
\Hom_k(Q_n,B) & \text{for } n>0,
\end{array} \right.$$
with differentiation given by
\begin{eqnarray*}
( \partial^0b)( \lambda ) & = & \lambda b-b \lambda \\
( \partial^nf)( \lambda_1 \otimes \cdots \otimes \lambda_{n+1} ) & = & \lambda_1 f( \lambda_2 \otimes \cdots \otimes \lambda_{n+1} ) \\
& + & \sum_{i=1}^n(-1)^if( \lambda_1 \otimes \cdots \otimes \lambda_i \lambda_{i+1} \otimes \cdots \otimes \lambda_{n+1} ) \\
& + & (-1)^{n+1} f( \lambda_1 \otimes \cdots \otimes \lambda_n ) \lambda_{n+1}.
\end{eqnarray*}
The cohomology of this complex is the \emph{Hochschild cohomology of $\La$, with coefficients in $B$}. We denote this by $\HH^* ( \La, B)$. The homological counterpart to Hochschild cohomology is defined using tensor product instead of the $\Hom$-functor. The \emph{Hochschild homology complex}
$$\cdots \to H_3 \xrightarrow{\partial_3} H_2 \xrightarrow{\partial_2} H_1 \xrightarrow{\partial_1} H_0 \to 0 \to 0 \to \cdots$$
is defined as follows:
$$H_n = \left \{
\begin{array}{ll}
0 & \text{for } n<0, \\
B & \text{for } n=0, \\
B \otimes_k Q_n & \text{for } n>0,
\end{array} \right.$$
with differentiation given by
\begin{eqnarray*}
\partial_n ( b \otimes \lambda_1 \otimes \cdots \otimes \lambda_n ) & = & b \lambda_1 \otimes \lambda_2 \otimes \cdots \otimes \lambda_n \\
& + & \sum_{i=1}^{n-1}(-1)^i b \otimes \lambda_1 \otimes \cdots \otimes \lambda_i \lambda_{i+1} \otimes \cdots \otimes \lambda_n \\
& = & (-1)^n \lambda_n b \otimes \lambda_1 \otimes \cdots \otimes \lambda_{n-1}.
\end{eqnarray*}
The homology of this complex is the \emph{Hochschild homology of $\La$, with coefficients in $B$}. We denote this by $\HH_* ( \La, B)$.

When the algebra $\La$ is projective as a module over the ground ring $k$, the Hochschild cohomology and homology groups can be interpreted using $\Ext$ and $\Tor$ over the enveloping algebra $\Lae$. Namely, for each non-negative integer $n$, let $P_n = Q_{n+2}$, that is, the $(n+2)$-fold tensor product of $\La$ over $k$. We endow $P_n$ with a left $\Lae$-module structure (that is, a bimodule structure) by defining
$$( \lambda \otimes \lambda' )( \lambda_0 \otimes \cdots \otimes \lambda_{n+1} ) = \lambda \lambda_0 \otimes \cdots \otimes \lambda_{n+1} \lambda',$$
and for each $n \ge 1$, define a bimodule homomorphism $P_n \xrightarrow{d_n} P_{n-1}$ by
$$\lambda_0 \otimes \cdots \otimes \lambda_{n+1} \mapsto \sum_{i=0}^n(-1)^i \lambda_0 \otimes \cdots \otimes \lambda_i \lambda_{i+1} \otimes \cdots \otimes \lambda_{n+1}.$$
The sequence
$$\mathbb{S} \colon \cdots \to P_3 \xrightarrow{d_3} P_2 \xrightarrow{d_2} P_1 \xrightarrow{d_1} P_0 \xrightarrow{\mu} \La \to 0$$
of bimodules and homomorphisms, where $\mu$ is the multiplication map, is exact (cf.\ \cite[p.\ 174-175]{CartanEilenberg}), and we denote by $\mathbb{S}_{\La}$ the complex obtained by deleting $\La$. Since $P_n$ and $\Lae \otimes_k Q_n$ are isomorphic as $\Lae$-modules, adjointness gives
$$\Hom_{\Lae}(P_n,B) \cong \Hom_k(Q_n, \Hom_{\Lae}( \Lae,B)) \cong \Hom_k(Q_n,B),$$
and the Hochschild cohomology complex is isomorphic to the complex $\Hom_{\Lae}( \mathbb{S}_{\La}, B)$ (where we view $B$ as a left $\Lae$-module). Similarly, the Hochschild homology complex is isomorphic to the complex $B \otimes_{\Lae} \mathbb{S}_{\La}$ (where we view $B$ as a right $\Lae$-module). Now, if $\La$ is projective as a module over $k$, then so is $Q_n$, hence the functor $\Hom_k(Q_n,-)$ is exact. By adjointness, this functor is isomorphic to the functor $\Hom_{\Lae}(P_n,-)$, and therefore $P_n$ is a projective bimodule. Thus the sequence $\mathbb{S}$ is a projective bimodule resolution of $\La$, giving isomorphisms
\begin{eqnarray*}
\HH^*( \La, B ) & \cong & \Ext_{\Lae}^*( \La, B) \\
\HH_*( \La, B ) & \cong & \Tor^{\Lae}_*( B, \La ).
\end{eqnarray*}

The Hochschild cohomology of an algebra lives only in positive degrees, as does the Hochschild homology. The focus of this paper is a (co)homological theory which extends the classical one. In order to give the definition, we recall some general notions from \cite{AvramovMartsinkovsky}. Suppose $\La$ is a two-sided Noetherian Gorenstein ring, say of Gorenstein dimension $d$. That is to say, the injective dimensions of $\La$, both as a left and as a right module over itself, are equal to $d$. Then every finitely generated left $\La$-module $M$ admits a \emph{complete resolution}
$$\mathbb{T} \colon \cdots \to T_2 \to T_1 \to T_0 \to T_{-1} \to T_{-2} \to \cdots,$$
i.e.\ an acyclic complex of finitely generated projective modules with the following properties
(see \cite[Theorem 3.2]{AvramovMartsinkovsky}):
\begin{enumerate}
\item the dual complex $\mathbb{T}^*$ is acyclic,
\item there exists a projective resolution $\mathbb{P}$ of $M$ and a chain map $\mathbb{T} \xrightarrow{f} \mathbb{P}$ with the property that $f_n$ is bijective for $n \ge d$.
\end{enumerate}
Property (2) implies that $\mathbb{T}$ is ``eventually" a projective resolution of $M$. Given another $\La$-module $N$ and an integer $n \in \mathbb{Z}$, the \emph{Tate cohomology group} $\Tateext_{\La}^n(M,N)$ is the $n$th cohomology of the complex $\Hom_{\La}( \mathbb{T},N)$. If $N$ is a right module, the \emph{Tate homology group} $\Tatetor^{\La}_n(N,M)$ is the $n$th homology of the complex $N \otimes_{\La} \mathbb{T}$. Naturally, the Tate (co)homology is independent of the complete resolution of $M$, and, in the homological case, it can be computed using a complete resolution of $N$ \cite{ChristensenJorgensen}. Moreover, by property (2) there are isomorphisms
\begin{eqnarray*}
\Tateext_{\La}^n(M,N) & \cong & \Ext_{\La}^n(M,N) \\
\Tatetor^{\La}_n(N,M) & \cong & \Tor^{\La}_n(N,M)
\end{eqnarray*}
for all $n \ge d+1$. The original cohomological definition is due to Tate, who introduced the cohomology groups for modules over the integral group ring of a finite group in order to study class field theory (cf.\ \cite[XII, \S 3]{CartanEilenberg}).

Having recalled the classical defnition of Tate cohomology and homology, we may now define the Hochschild cohomological and homological versions.
\begin{definition}
Let $k$ be a commutative ring and $\La$ a $k$-algebra such that the enveloping algebra $\Lae$ is two-sided Noetherian and Gorenstein. For an integer $n \in \mathbb{Z}$ and a bimodule $B$, the $n$th \emph{Tate-Hochschild cohomology} group $\TateHH^n ( \La, B)$ and the $n$th \emph{Tate-Hochschild homology} group $\TateHH_n ( \La, B)$ are defined as
\begin{eqnarray*}
\TateHH^n( \La, B ) & = & \Tateext_{\Lae}^n( \La, B) \\
\TateHH_n( \La, B ) & = & \Tatetor^{\Lae}_n( B, \La ).
\end{eqnarray*}
\end{definition}
Note that if the Gorenstein dimension of the enveloping algebra is $d$, then for every $n \ge d+1$ there are isomorphisms $\TateHH^n( \La, B ) \cong \Ext_{\Lae}^n( \La, B)$ and $\TateHH_n( \La, B ) \cong \Tor^{\Lae}_n( B, \La )$. In particular, when $\La$ is projective as a $k$-module, then there are isomorphisms
\begin{eqnarray*}
\TateHH^n( \La, B ) & \cong & \HH^n( \La, B ) \\
\TateHH_n( \La, B ) & \cong & \HH_n( \La, B )
\end{eqnarray*}
whenever $n \ge d+1$. A special case appears when the enveloping algebra is two-sided Noetherian and selfinjective. By definition, the enveloping algebra is then of Gorenstein dimension zero, and the Tate-Hochschild (co)homology groups are therefore defined and agree with the classical Hochschild (co)homology groups in all positive degrees. In particular, this is the case for finite dimensional Frobenius algebras (see the next section); for such algebras, the Tate-Hochschild (co)homology agrees with the \emph{stable Hochschild (co)homology} introduced in \cite{EuSchedler}.

We shall mainly be working with finite dimensional $k$-algebras (hence $k$ will be a field), hence the requirement (in the definition of Gorenstein algebras) that the enveloping algebra be two-sided Noetherian is unnecessary. In other words, a finite dimensional algebra is Gorenstein if and only if its injective dimensions as a left and right module over itself are finite. It is known that in this case the two injective dimensions are the same. The following result shows that if a finite dimensional algebra is Gorenstein, then so is its enveloping algebra. We include a proof due to the lack of a reference. Consequently, Tate-Hochschild (co)homology is defined for finite dimensional Gorenstein algebras. Note that the result shows in particular that the enveloping algebra of a selfinjective algebra is again selfinjective.

\begin{lemma}\label{tensorgorenstein}
If $k$ is a field and $\La$ and $\Gamma$ are finite dimensional Gorenstein $k$-algebras of Gorenstein dimensions $s$ and $t$, respectively, then their tensor product $\La \otimes_k \Gamma$ is Gorenstein of Gorenstein dimension at most $s+t$. In particular, the enveloping algebra $\Lae$ is Gorenstein of Gorenstein dimension at most $2s$.
\end{lemma}

\begin{proof}
Choose injective resolutions
$$0 \to \La \to I^0_{\La} \to \cdots \to I^s_{\La} \to 0$$
and
$$0 \to \Gamma \to I^0_{\Gamma} \to \cdots \to I^t_{\Gamma} \to 0$$
over $\La$ and $\Gamma$, respectively, both as left modules. When we delete the algebras and tensor the resulting complexes over $k$, we obtain a complex
$$\mathbb{E} \colon 0 \to E^0 \to E^1 \to \cdots \to E^{s+t} \to 0$$
in which $E^n = \oplus_{j=0}^n (I^j_{\La} \otimes_k I^{n-j}_{\Gamma})$. In general, if $I_{\La}$ and $I_{\Gamma}$ are injective left modules over $\La$ and $\Gamma$, respectively, then the right modules $D(I_{\La})$ and $D(I_{\Gamma})$ are projective, and so $D(I_{\La}) \otimes_k D(I_{\Gamma})$ is a projective right ($\La \otimes_k \Gamma$)-module. But this   right ($\La \otimes_k \Gamma$)-module is isomorphic to $D(I_{\La} \otimes_k I_{\Gamma})$, and consequently the left ($\La \otimes_k \Gamma$)-module $I_{\La} \otimes_k I_{\Gamma}$ is injective. This shows that the complex $\mathbb{E}$ is an injective resolution of $\La \otimes_k \Gamma$ as a left module over itself. Similarly, by starting with injective resolutions of right modules, we end up with an injective resolution (of length $s+t$) of $\La \otimes_k \Gamma$ as a right module over itself. This proves the first part of the lemma. The second part follows immediately, since the opposite algebra of a Gorenstein algebra is also Gorenstein of the same dimension.
\end{proof}

Note also that when $\La$ is finite dimensional algebra and $B$ is a $\La$-$\La$-bimodule which is finitely
generated as either a left or right $\La$-module, then the Tate-Hochschild homology $\TateHH_n(\La,B)$
and cohomology $\TateHH^n(\La,B)$ are just finite dimensional vector spaces over $k$ for all
$n\in\mathbb Z$.

The main results in this section establish Tate-Hochschild duality isomorphisms for Gorenstein algebras. These results follow from a more general duality result for Tate homology, which we prove after the following two lemmas. The first lemma is well known in the case of ordinary (co)homology: over any finite dimensional algebra $\Gamma$ there is an isomorphism
$$D \left ( \Ext_{\Gamma}^i(X,Y) \right ) \cong \Tor^{\Gamma}_i(D(Y),X)$$
for all $i \ge 0$ and all modules $X,Y$ (cf.\ \cite[VI, Proposition 5.3]{CartanEilenberg}).

\begin{lemma}\label{dualityexttor}
Let $\La$ be a finite dimensional Gorenstein algebra and $M$ and $N$ two left $\La$-modules, with $M$ finitely generated. Then there is an isomorphism
$$D \left ( \Tateext_{\La}^n( M,N ) \right ) \cong \Tatetor^{\La}_n( D(N),M )$$
for all $n \in \mathbb{Z}$. In particular, if $B$ is a bimodule, then there is an isomorphism
$$D \left ( \TateHH^n( \La, B ) \right ) \cong \TateHH_n( \La, D(B) )$$
for all $n \in \mathbb{Z}$.
\end{lemma}

\begin{proof}
Let $\mathbb{T}$ be a complete resolution of $M$, and for each $i \in \mathbb{Z}$, denote by $\Omega_{\La}^i ( \mathbb{T} )$ the image of the $i$th differential in $\mathbb{T}$. Fix $n\in\mathbb Z$, and denote the Gorenstein dimension of $\La$ by $d$. Let $m$ be any integer with the property that $m+n>d$. Then there are isomorphisms
\begin{eqnarray*}
D \left (  \Tateext_{\La}^n( M,N ) \right ) & \cong & D \left ( \Ho^n( \Hom_{\La}( \mathbb{T},N ) ) \right ) \\
& \cong & D \left (  \Tateext_{\La}^{n+m}( \Omega_{\La}^{-m} ( \mathbb{T} ),N ) \right ) \\
& \cong & D \left ( \Ext_{\La}^{n+m}( \Omega_{\La}^{-m} ( \mathbb{T} ),N ) \right ) \\
& \cong & \Tor^{\La}_{n+m}( D(N), \Omega_{\La}^{-m} ( \mathbb{T} ) ) \\
& \cong & \Tatetor^{\La}_{n+m}( D(N), \Omega_{\La}^{-m} ( \mathbb{T} ) ) \\
& \cong & \Ho_n( D(N) \otimes_{\La} \mathbb{T} )\\
& \cong & \Tatetor^{\La}_{n}( D(N), M ),
\end{eqnarray*}
and we have proved the first part. The second part follows from the first and the definition of Tate-Hochschild (co)homology.
\end{proof}

The second lemma seems to be well known; it is a special case of \cite[Proposition 20.10]{AndersonFuller}. We include a proof.

\begin{lemma}\label{dualityprojective}
Let $\La$ be any ring and $M$ a left $\La$-module. If $P$ is a finitely generated projective left $\La$-module, then there is an isomorphism
$$\psi_P \colon \Hom_{\La}(P, \La ) \otimes_{\La} M \to \Hom_{\La} (P,M)$$
given by $\psi_P(f \otimes m)(p) = f(p)m$. This isomorphism is natural in $P$.
\end{lemma}

\begin{proof}
The map $\psi_P$ is well defined since the pairing
\begin{eqnarray*}
\Hom_{\La}(P, \La ) \times M & \xrightarrow{\theta} & \Hom_{\La} (P,M) \\
(f,m) & \mapsto & (p \mapsto f(p)m )
\end{eqnarray*}
satisfies $\theta ( f \lambda, m ) = \theta (f, \lambda m )$ for all $\lambda \in \La$. When $P = \La$, this map is just the composition of the isomorphisms
$$\Hom_{\La}( {_{\La}\La}, {_{\La}\La}_{\La} ) \otimes_{\La} M \to \La_{\La} \otimes_{\La} M \to M \to \Hom_{\La} ( {_{\La}\La},M)$$
and hence an isomorphism itself. Extending to the case when $P$ is a finitely generated free module, and then to the case when $P$ is a summand of such a module, we see that the first half of the lemma holds.

As for the naturality in $P$, let $P_1 \xrightarrow{h} P_2$ be a map between finitely generated projective left $\La$-modules, and consider the diagram
$$\xymatrix{
\Hom_{\La}(P_2, \La ) \otimes_{\La} M \ar[d]^{\psi_{P_2}} \ar[r]^{h^* \otimes 1_M} & \Hom_{\La}(P_1, \La ) \otimes_{\La} M \ar[d]^{\psi_{P_1}} \\
\Hom_{\La} (P_2,M) \ar[r]^{h^*} & \Hom_{\La} (P_1,M) }$$
If $f \in \Hom_{\La}(P_2, \La ), m \in M$ and $p \in P_1$, then
\begin{eqnarray*}
\left [ (h^* \circ \psi_{P_2})(f \otimes m) \right ] (p) & = & ( \psi_{P_2}(f \otimes m) \circ h) (p) \\
& = & \psi_{P_2} (f \otimes m) \left ( h(p) \right ) \\
& = & f \left ( h(p) \right ) m \\
& = & (f \circ h)(p)m \\
& = & \psi_{P_1}(f \circ h \otimes m) (p) \\
& = & \left [ ( \psi_{P_1} \circ (h^* \otimes 1_M))(f \otimes m) \right ] (p),
\end{eqnarray*}
hence the diagram commutes.
\end{proof}

We are now ready to prove the general duality result for Tate homology and cohomology. Recall first that when $k$ is a field and $\La$ is a finite dimensional $k$-algebra, then every finitely generated module $M$ admits a \emph{minimal projective resolution}
$$\cdots \to P_2 \xrightarrow{d_2} P_1 \xrightarrow{d_1} P_0 \xrightarrow{d_0} M \to 0.$$
This projective resolution appears as a direct summand of every projective resolution of $M$, and it is unique up to isomorphism. For every $n \ge 0$, the $n$th \emph{syzygy} of $M$, denoted $\Omega_{\La}^n(M)$, is the image of the map $d_n$.

\begin{theorem}\label{dualitygeneral}
Let $\La$ be a finite dimensional Gorenstein algebra, $M,L$ two finitely generated left modules, and $N$ a finitely generated right module. If the Gorenstein dimension of $\La$ is at most $d$, then there are vector space isomorphisms
\begin{eqnarray*}
\Tatetor^{\La}_n(N,M) & \cong & \Tatetor^{\La}_{-(n-d+1)}( \Omega_{\La}^d(M)^*, D(N)) \\
\Tateext_{\La}^n(M,L) & \cong & \Tateext_{\La}^{-(n-d+1)}( L,D( \Omega_{\La}^d(M)^*))
\end{eqnarray*}
for all $n \in \mathbb{Z}$.
\end{theorem}

\begin{proof}
Consider the minimal projective resolution
$$\cdots \to P_2 \to P_1 \to P_0 \to M \to 0$$
of $M$. It follows from \cite[Lemma 2.5 and Construction 3.6]{AvramovMartsinkovsky} that $M$ admits a complete resolution
$$\mathbb{T} \colon \cdots \to T_2 \xrightarrow{\partial_2} T_1 \xrightarrow{\partial_1} T_0 \xrightarrow{\partial_0} T_{-1} \xrightarrow{\partial_{-1}} T_{-2} \to \cdots$$
such that there exists a chain map
$$\xymatrix{
\cdots \ar[r] & T_2 \ar[r]^{\partial_2} \ar[d]^{f_2} & T_1 \ar[r]^{\partial_1} \ar[d]^{f_1} & T_0 \ar[r]^{\partial_0} \ar[d]^{f_0} & T_{-1} \ar[r]^{\partial_{-1}} \ar[d]^{f_{-1}} & T_{-2} \ar[d]^{f_{-2}} \ar[r] & \cdots \\
\cdots \ar[r] & P_2 \ar[r] & P_1 \ar[r] & P_0 \ar[r] & 0 \ar[r] & 0 \ar[r] & \cdots }$$
in which $f_n$ is bijective for $n \ge d$. Consequently the image of the map $\partial_d$ is isomorphic to $\Omega_{\La}^d(M)$; we denote this module by $X$. We must show that
$$\Tatetor^{\La}_n(N,X) \cong \Tatetor^{\La}_{-(n+1)}( X^*, D(N))$$
for all $n$, since $\Tatetor^{\La}_n(N,X)$ is isomorphic to $\Tatetor^{\La}_{n+d}(N,M)$.

By adjointness, the complexes $\Hom_k( N \otimes_{\La} \mathbb{T}, k )$ and $\Hom_{\La} ( \mathbb{T}, \Hom_k(N,k) )$ are isomorphic, that is, there is an isomorphism
$$D( N \otimes_{\La} \mathbb{T} ) \cong \Hom_{\La} ( \mathbb{T}, D(N) )$$
of complexes. Moreover, by Lemma \ref{dualityprojective}, there is an isomorphism
$$\xymatrix{
\cdots \to \Hom_{\La} ( T_{n-1} , D(N) ) \ar[d]^{\psi_{T_{n-1}}^{-1}} \ar[r]^{\partial_n} & \Hom_{\La} ( T_{n} , D(N) ) \ar[d]^{\psi_{T_{n}}^{-1}}  \ar[r] & \cdots \\
 \cdots \to \Hom_{\La} ( T_{n-1}, \La ) \otimes_{\La} D(N) \ar[r]^{\partial_n} & \Hom_{\La} ( T_{n}, \La ) \otimes_{\La} D(N)  \ar[r] & \cdots }$$
between the complexes $\Hom_{\La} ( \mathbb{T}, D(N) )$ and $\mathbb{T}^* \otimes_{\La} D(N)$. In general, note that if $C$ is a complex of finite dimensional vector spaces over $k$, then $\Ho_n(C)$ and
$\Ho_{-n}(D(C))$ have the same dimension, and are therefore isomorphic as vector spaces.  This explains
the second isomorphism below. Now since $\mathbb{T}^*$ is a complete resolution of $X^*$, we see that
\begin{eqnarray*}
\Tatetor^{\La}_n(N,X) & \cong & \Ho_{n+d} ( N \otimes_{\La} \mathbb{T} ) \\
& \cong & \Ho_{-(n+d)} \left ( D( N \otimes_{\La} \mathbb{T} ) \right ) \\
& \cong & \Ho_{-(n+d)} \left ( \Hom_{\La} ( \mathbb{T}, D(N) ) \right ) \\
& \cong & \Ho_{-(n+d)} \left ( \mathbb{T}^* \otimes_{\La} D(N) \right ) \\
& \cong & \Tatetor^{\La}_{-(n+1)}( X^*, D(N))
\end{eqnarray*}
and the proof of the homology part is complete.

For the cohomology part, we use Lemma \ref{dualityexttor} twice, together with the homology part we just proved:
\begin{eqnarray*}
D \left ( \Tateext_{\La}^n( M,L ) \right ) & \cong & \Tatetor^{\La}_n( D(L),M ) \\
& \cong & \Tatetor^{\La}_{-(n-d+1)}( \Omega_{\La}^d(M)^*, D^2(L)) \\
& \cong & \Tatetor^{\La}_{-(n-d+1)}( \Omega_{\La}^d(M)^*, L) \\
& \cong & D \left ( \Tateext_{\La}^{-(n-d+1)}( L,D( \Omega_{\La}^d(M)^*))\right ).
\end{eqnarray*}
Hence $\Tateext_{\La}^n( M,L )$ and $\Tateext_{\La}^{-(n-d+1)}( L,D( \Omega_{\La}^d(M)^*))$ are isomorphic.
\end{proof}

We can now prove the duality result for Tate-Hochschild (co)homology; this is just a direct application of Theorem \ref{dualitygeneral}.

\begin{theorem}\label{dualityTHgorenstein}
If $\La$ is a finite dimensional Gorenstein algebra of Gorenstein dimension $d$, and $B$ is a $\La$-$\La$-bimodule which is finitely generated as either a left or right $\La$-module, then
there are isomorphisms of vector spaces
\begin{eqnarray*}
\TateHH_n( \La, B ) & \cong & \Tatetor^{\Lae}_{-(n-2d+1)}(\Omega_{\Lae}^{2d}( \La )^* , D(B)) \\
\TateHH^n( \La, B ) & \cong & \Tateext_{\Lae}^{-(n-2d+1)} \left ( B, D \left ( \Omega_{\Lae}^{2d}( \La )^* \right ) \right )
\end{eqnarray*}
for all $n \in \mathbb{Z}$, where $(-)^* = \Hom_{\Lae}(-, \Lae )$. In particular, there are isomorphisms
\begin{eqnarray*}
\TateHH_n( \La, \La ) & \cong & \Tatetor^{\Lae}_{-(n-2d+1)}( \Omega_{\Lae}^{2d}( \La )^*, D( \La ) ) \\
\TateHH^n( \La, \La ) & \cong & \TateHH^{-(n-2d+1)} \left ( \La , D \left ( \Omega_{\Lae}^{2d}( \La )^* \right ) \right )
\end{eqnarray*}
for all $n \in \mathbb{Z}$.
\end{theorem}

\begin{proof}
By Lemma \ref{tensorgorenstein}, the enveloping algebra $\Lae$ is Gorenstein of dimension at most $2d$, hence the isomorphisms in the first part follow immediately from Theorem \ref{dualitygeneral}:
$$\TateHH_n( \La, B ) =  \Tatetor^{\Lae}_n( B, \La ) \cong \Tatetor^{\Lae}_{-(n-2d+1)}( \Omega_{\Lae}^{2d}( \La )^*, D(B) )$$
$$\TateHH^n( \La, B )  =  \Tateext_{\Lae}^n( \La, B ) \cong \Tateext_{\Lae}^{-(n-2d+1)} \left ( B, D \left ( \Omega_{\Lae}^{2d}( \La )^* \right ) \right ).$$
The last part of the theorem follows directly from the first.
\end{proof}

We end this section by specializing to selfinjective algebras. Such and algebra is by definition Gorenstein, and its Gorenstein dimension is zero. Therefore, for this class of algebras, Theorem \ref{dualityTHgorenstein} takes the following form.

\begin{theorem}\label{dualityTHselfinjective}
If $\La$ is a finite dimensional selfinjective algebra, and $B$ is a bimodule, then there are isomorphisms
of vector spaces
\begin{eqnarray*}
\TateHH_n( \La, B ) & \cong & \Tatetor^{\Lae}_{-(n+1)}( \La^*, D(B) ) \\
\TateHH^n( \La, B ) & \cong & \Tateext_{\Lae}^{-(n+1)} \left ( B, D( \La^*) \right )
\end{eqnarray*}
for all $n \in \mathbb{Z}$, where $(-)^* = \Hom_{\Lae}(-, \Lae )$. In particular, there are isomorphisms
\begin{eqnarray*}
\TateHH_n( \La, \La ) & \cong & \Tatetor^{\Lae}_{-(n+1)}( \La^*, D( \La ) ) \\
\TateHH^n( \La, \La ) & \cong & \TateHH^{-(n+1)} \left ( \La , D( \La^* ) \right )
\end{eqnarray*}
for all $n \in \mathbb{Z}$.
\end{theorem}

\section{Frobenius algebras}\label{frobeniusalg}

In this section, we apply the Tate-Hochschild duality results from the last section to a special class of selfinjective algebras. Recall that a finite dimensional algebra $\La$ is \emph{Frobenius} if $\La$ and $D(\La)$ are
isomorphic as left $\La$-modules, and \emph{symmetric} if they are
isomorphic as bimodules. Suppose $\La$ is Frobenius, and fix an
isomorphism $\phi \colon \La \to D( \La )$ of left modules. Let $y
\in \La$ be any element, and consider the linear functional $\phi(1)
\cdot y \in D(\La)$. This is the $k$-linear map $\La \to k$ defined
by $\lambda \mapsto \phi(1)(y \lambda)$, where $k$ is the ground field. Since $\phi$ is surjective,
there is an element $x \in \La$ having the property that $\phi(x) =
\phi(1) \cdot y$, giving $x \cdot \phi(1) = \phi(1) \cdot y$ since
$\phi$ is a map of left $\La$-modules. The map $y \mapsto x$ defines
a $k$-algebra automorphism on $\La$, and its inverse $\nu$ is the
\emph{Nakayama automorphism} of $\La$ (with respect to $\phi$). Thus
$\nu$ is defined by $\phi(1)(\lambda x) = \phi(1)( \nu(x) \lambda)$
for all $\lambda, x \in \La$. The Nakayama automorphism is unique up
to an inner automorphism. Namely, if $\phi' \colon \La \to D( \La )$
is another isomorphism of left modules yielding a Nakayama
automorphism $\nu'$, then there exists an invertible element $z \in
\La$ such that $\nu = z \nu' z^{-1}$.
Note that $\La$ is symmetric if and only if the Nakayama
automorphism is the identity.

Since $D( \La )$ is an injective left $\La$-module, a Frobenuis
algebra is always left selfinjective. However, the definition is
left-right symmetric. For if $\phi \colon _{\La}\La \to
D(\La_{\La})$ is an isomorphism of left $\La$-modules, we can
dualize and obtain an isomorphism $D( \phi ) \colon D^2( \La_{\La}
) \to D( _{\La}\La )$ of right modules. Composing with the natural
isomorphism $\La_{\La} \cong D^2( \La_{\La} )$, we obtain an
isomorphism $\La_{\La} \to D( _{\La}\La)$ of right $\La$-modules. This left-right symmetry implies that the opposite algebra of a Frobenius algebra is also Frobenius, and that its Nakayama automorphism is the inverse of the original one.

\begin{lemma}\label{frobeniusopposite}
If $\La$ is a Frobenius algebra with a Nakayama automorphism
$\nu$, then $\La^{\op}$ is Frobenius with $\nu^{-1}$ as a Nakayama
automorphism.
\end{lemma}

\begin{proof}
As seen above, the definition of a Frobenius algebra is left-right
symmetric. Moreover, an isomorphism $\La \to D( \La )$ of right
$\La$-modules may be viewed as an isomorphism $\La^{\op} \to D(
\La^{\op} )$ of left $\La^{\op}$-modules. Hence $\La$ is Frobenius
if and only if $\La^{\op}$ is.

Now suppose $\La$ is Frobenius, and let $\phi \colon \La \to D( \La
)$ be an isomorphism of left modules with corresponding Nakayama
automorphism $\nu \colon \La \to \La$. The composition of
isomorphisms
$$\La \to D^2( \La ) \xrightarrow{D( \phi )} D( \La )$$
of right $\La$-modules can then be viewed as an isomorphism
$\phi^{\op}$ of left $\La^{\op}$-modules. Thus $\phi^{\op}(1)(
\lambda ) = \phi ( \lambda )(1)$ for all $\lambda \in \La^{\op}$.
Denote the multiplication of two elements $x$ and $y$ in $\La^{\op}$
by $x \cdot y$, so that $x \cdot y = yx$, where $yx$ is the ordinary
product in $\La$. Then
\begin{eqnarray*}
\phi^{\op}(1)( \lambda \cdot x ) & = & \phi^{\op}(1)( x \lambda ) \\
& = & \phi (x \lambda )(1) \\
& = & \phi ( \lambda \nu^{-1}(x))(1) \\
& = & \phi^{\op}(1)( \lambda \nu^{-1}(x)) \\
& = & \phi^{\op}(1)( \nu^{-1}(x) \cdot \lambda)
\end{eqnarray*}
for all $\lambda, x \in \La^{\op}$, hence $\nu^{-1}$ is a Nakayama
automorphism for $\La^{\op}$.
\end{proof}

The tensor product of two Frobenius algebras is also Frobenius, with the obvious Nakayama automorphism. We record this in the following lemma.

\begin{lemma}\label{frobeniustensor}
If $\La$ and $\Gamma$ are Frobenius $k$-algebras with Nakayama
automorphisms $\nu_{\La}$ and $\nu_{\Gamma}$, respectively, then
$\La \otimes_k \Gamma$ is Frobenius with $\nu_{\La} \otimes
\nu_{\Gamma}$ as a Nakayama automorphism.
\end{lemma}

\begin{proof}
Let $\phi_{\La} \colon \La \to D( \La )$ and $\phi_{\Gamma} \colon
\Gamma \to D( \Gamma )$ be isomorphisms of left modules, with
corresponding Nakayama automorphisms $\nu_{\La}$ and
$\nu_{\Gamma}$, respectively. Then the composition
$$\La \otimes_k \Gamma \xrightarrow{\phi_{\La} \otimes \phi_{\Gamma}} D( \La ) \otimes_k
D( \Gamma ) \to D( \La \otimes_k \Gamma )$$ of left $( \La
\otimes_k \Gamma )$-isomorphisms shows that $\La \otimes_k \Gamma$
is Frobenius.

Denote this composition by $\phi$. Then
\begin{eqnarray*}
\phi ( 1 \otimes 1 ) \left ( [ \lambda \otimes \gamma ][x \otimes
y] \right ) &
= & \phi ( 1 \otimes 1 ) ( \lambda x \otimes \gamma y ) \\
& = & \phi_{\La} ( \lambda x ) \phi_{\Gamma} ( \gamma y ) \\
& = & \phi_{\La} ( \nu_{\La}(x) \lambda ) \phi_{\Gamma} ( \nu_{\Gamma}(y) \gamma ) \\
& = & \phi ( 1 \otimes 1 ) ( \nu_{\La}(x) \lambda \otimes
\nu_{\Gamma}(y) \gamma ) \\
& = & \phi ( 1 \otimes 1 ) ( [ \nu_{\La}(x) \otimes
\nu_{\Gamma}(y) ][ \lambda \otimes \gamma ] )
\end{eqnarray*}
for all $\lambda \otimes \gamma$ and $x \otimes y$ in $\La \otimes
\Gamma$. This shows that $\nu_{\La} \otimes \nu_{\Gamma}$ is a
Nakayama automorphism for $\La \otimes \Gamma$.
\end{proof}

Combining Lemma \ref{frobeniusopposite} and Lemma \ref{frobeniustensor}, we see that the enveloping algebra of a Frobenius algebra is again Frobenius.

\begin{corollary}\label{frobeniusenveloping}
If $\La$ is a Frobenius algebra with a Nakayama automorphism
$\nu$, then $\Lae$ is Frobenius with $\nu \otimes \nu^{-1}$ as a
Nakayama automorphism.
\end{corollary}

The Tate-Hochschild duality results below for Frobenius algebras involve twisted modules. If $\La$ is an arbitrary ring and $\La \xrightarrow{f} \La$ is an automorphism, then we can endow a $\La$-module $M$ with a new $\La$-module structure as follows: for $\lambda \in \La$ and $m \in M$, let $\lambda \cdot m = f( \lambda )m$. We denote this twisted $\La$-module by ${_fM}$. If $B$ is a bimodule over $\La$, and $\La \xrightarrow{g} \La$ another automorphism, then we can twist on both sides and obtain a bimodule ${_fB}_g$. A special case is the bimodule ${{_f}\La}_g$, which is isomorphic to ${_{g^{-1}}\La}_{f^{-1}}$ when the two automorphisms $f$ and $g$ commute. In particular, the bimodules ${{_f}\La}_1$ and ${{_1}\La}_{f^{-1}}$ are isomorphic, and so are ${{_f}\La}_f$ and $\La$ itself. Note that the twisted module ${_fM}$ is isomorphic to ${{_f}\La}_1 \otimes_{\La} M$.

Suppose now, as before, that $\La$ is Frobenius with an isomorphism $\phi \colon \La \to D( \La )$ of left modules, and let $\nu$ be a corresponding Nakayama automorphism. Then $\phi$ is an isomorphism between the bimodules $_1\La_{\nu^{-1}}$ and $D(\La)$, and from above we see that $D(\La)$ is also isomorphic to ${_{\nu}\La}_1$. We can use this to show that the ring dual of a $\La$-module is just the $k$-dual twisted by $\nu$.

\begin{lemma}\label{ringdual}
If $\La$ is a Frobenius algebra with a Nakayama automorphism
$\nu$, then for any finitely generated left module $M$, the right modules $M^*$ and $D(M)_{\nu}$ are isomorphic.
\end{lemma}

\begin{proof}
Standard $\Hom$-tensor adjunction gives
\begin{eqnarray*}
M^* & = & \Hom_{\La}(M, \La ) \\
& \cong & \Hom_{\La}(M, D^2( \La ) ) \\
& = & \Hom_{\La}(M, \Hom_k(D( \La ),k) ) \\
& \cong & \Hom_k( D( \La ) \otimes_{\La} M,k ) \\
& \cong & \Hom_k( {_{\nu}\La}_1 \otimes_{\La} M,k ) \\
& \cong & \Hom_k( {_{\nu}M},k ) \\
& = & D( {_{\nu}M} ).
\end{eqnarray*}
Since $D(M)_{\nu} = D( {_{\nu}M} )$, the result follows.
\end{proof}

Using this lemma, Theorem \ref{dualitygeneral} takes the following form for modules over a Frobenius algebra.

\begin{theorem}\label{dualitygeneralfrobenius}
Let $\La$ be a Frobenius algebra with a Nakayama automorphism $\nu$, $M,L$ two finitely generated left modules, and $N$ a finitely generated right module. Then there are isomorphisms of vector spaces
\begin{eqnarray*}
\Tatetor^{\La}_n(N,M) & \cong & \Tatetor^{\La}_{-(n+1)}( D(M)_{\nu}, D(N)) \\
\Tateext_{\La}^n(M,L) & \cong & \Tateext_{\La}^{-(n+1)}( L, {_{\nu}M} )
\end{eqnarray*}
for all $n \in \mathbb{Z}$.
\end{theorem}

\begin{proof}
The homology isomorphism is obtained directly by combining Theorem \ref{dualitygeneral} with Lemma \ref{ringdual}, and so does the cohomology isomorphism, when noting that there are isomorphisms
$$D( M^* ) \cong D \left ( D(M)_{\nu} \right ) = {_{\nu}D^2(M)} \cong {_{\nu}M}$$
of left $\La$-modules.
\end{proof}

Before applying this to Tate-Hochschild (co)homology, we include a result which shows the following: if one of the modules in a Tate homology group is twisted by an automorphism, then we may instead twist the other module by the inverse.

\begin{lemma}\label{homologytwist}
Let $\La$ be a ring with an automorphism $\La \xrightarrow{f} \La$, and $M$ and $N$ a right and a left $\La$-module, respectively. Then there is an isomorphism $\Tor^{\La}_n( M_f,N) \cong \Tor^{\La}_n( M, {_{f^{-1}}N})$ for every $n \ge 0$. If in addition $\La$ is a finite dimensional Gorenstein algebra, and $M$ is finitely generated, then there are isomorphisms $\Tatetor^{\La}_n( M_f,N) \cong \Tatetor^{\La}_n( M, {_{f^{-1}}N})$ for every $n \in \mathbb{Z}$.
\end{lemma}

\begin{proof}
For the first part, note that the map $M_f \otimes_{\La} N \to M \otimes_{\La} {_{f^{-1}}N}$ given by $m \otimes n \mapsto m \otimes n$ is well defined, and therefore an isomorphism. Moreover, if $P$ is a projective right $\La$-module, then so is $P_f$, and the twisting operation is exact. Therefore, if $\mathbb{P}$ is a projective resolution of $M$, then $\mathbb{P}_f$ is a projective resolution of $M_f$, giving
$$\Tor^{\La}_n( M_f,N) \cong \Ho_n \left ( \mathbb{P}_f \otimes _{\La} N \right ) \cong \Ho_n \left ( \mathbb{P} \otimes _{\La} {_{f^{-1}}N} \right ) \cong \Tor^{\La}_n( M, {_{f^{-1}}N}).$$
Suppose now that $\La$ is a finite dimensional Gorenstein algebra, and $M$ is finitely generated. If $\mathbb{T}$ is a complete resolution of $M$, then by definition there exists a projective resolution $\mathbb{P}$ of $M$ and a chain map $\mathbb{T} \xrightarrow{h} \mathbb{P}$ with the property that $h_n$ is bijective for $n \ge d$. Twisting by $f$, we see that $h$ is also a chain map $\mathbb{T}_f \xrightarrow{h} \mathbb{P}_f$. Moreover, we know that $\mathbb{P}_f$ is a projective resolution of $M_f$, and that the complex $\mathbb{T}_f$ is acyclic and consists of finitely generated projective modules. Now, if $X$ is an arbitrary right $\La$-module and $Y$ is a bimodule, then the map $\Hom_{\La}(X_f,Y) \to \Hom_{\La}(X,Y_{f^{-1}})$ given by $g \mapsto g$ is an isomorphism of left $\La$-modules. Therefore
$$( \mathbb{T}_f )^* = \Hom_{\La}( \mathbb{T}_f, \La ) \cong \Hom_{\La}( \mathbb{T}, {_1\La}_{f^{-1}} ) \cong \Hom_{\La} ( \mathbb{T}, {_f\La}_1 ) \cong {_f( \mathbb{T}^* )},$$
hence $( \mathbb{T}_f )^*$ is also acyclic. Consequently $\mathbb{T}_f$ is a complete resolution of $M_f$, giving
$$\Tatetor^{\La}_n( M_f,N) \cong \Ho_n \left ( \mathbb{T}_f \otimes _{\La} N \right ) \cong \Ho_n ( \mathbb{T} \otimes _{\La} {_{f^{-1}}N} ) \cong \Tatetor^{\La}_n( M, {_{f^{-1}}N}).$$
This completes the proof.
\end{proof}

We may now prove the Tate-Hochschild duality result for Frobenius algebras. The duality for the Tate-Hochschild homology $\TateHH_n( \La, \La )$ was proved by Eu and Schedler in \cite{EuSchedler}.

\begin{theorem}\label{dualityTHfrobenius}
If $\La$ is a Frobenius algebra with a Nakayama automorphism $\nu$, and $B$ a bimodule, then there are isomorphisms
\begin{eqnarray*}
\TateHH_n( \La, B )& \cong & \Tatetor^{\Lae}_{-(n+1)}( \La, {_{\nu^{-1}}D(B)}_1 ) \\
\TateHH^n( \La, B )& \cong & \Tateext_{\Lae}^{-(n+1)} ( B, {_{\nu^2}\La}_1 )
\end{eqnarray*}
for all $n \in \mathbb{Z}$. In particular, there are isomorphisms
\begin{eqnarray*}
\TateHH_n( \La, \La )& \cong & \TateHH_{-(n+1)}( \La, \La ) \\
\TateHH^n( \La, \La )& \cong & \TateHH^{-(n+1)} ( \La , {_{\nu^2}\La}_1 )
\end{eqnarray*}
for all $n \in \mathbb{Z}$.
\end{theorem}

\begin{proof}
For the homology isomorphism, we use Theorem \ref{dualitygeneralfrobenius} together with Lemma \ref{homologytwist}:
\begin{eqnarray*}
\TateHH_n( \La, B ) & = & \Tatetor_n^{\Lae}( B, \La )\\
& \cong & \Tatetor_{-(n+1)}^{\Lae}( D( \La )_{\nu_{\Lae}}, D(B) ) \\
& \cong & \Tatetor_{-(n+1)}^{\Lae}( ( {_{\nu}\La}_1)_{( \nu \otimes \nu^{-1} )}, D(B) ) \\
& \cong & \Tatetor_{-(n+1)}^{\Lae}( \La_{( \nu \otimes 1 )}, D(B) ) \\
& \cong & \Tatetor_{-(n+1)}^{\Lae}( \La, {_{( \nu \otimes 1 )^{-1}}D(B)} ) \\
& = & \Tatetor_{-(n+1)}^{\Lae}( \La, {_{\nu^{-1}}D(B)}_1 ).
\end{eqnarray*}
For the cohomology isomorphism, we use Theorem \ref{dualitygeneralfrobenius} directly:
\begin{eqnarray*}
\TateHH^n( \La, B ) & = & \Tateext_{\Lae}^n( \La, B ) \\
& \cong & \Tateext_{\Lae}^{-(n+1)} ( B, {_{\nu_{\Lae}}\La} ) \\
& = & \Tateext_{\Lae}^{-(n+1)} ( B, {_{\nu \otimes \nu^{-1}}\La} ) \\
& = & \Tateext_{\Lae}^{-(n+1)} ( B, {_{\nu}\La}_{\nu^{-1}} ) \\
& \cong & \Tateext_{\Lae}^{-(n+1)} ( B, {_{\nu^2}\La}_1 )
\end{eqnarray*}
When $B= \La$, then the isomorphism $\TateHH^n( \La, \La ) \cong \TateHH^{-(n+1)} ( \La , {_{\nu^2}\La}_1 )$ follows directly, whereas the isomorphism $\TateHH_n( \La, \La ) \cong \TateHH_{-(n+1)}( \La, \La )$ follows from the fact that $D( \La ) \cong {_{\nu}\La}_1$.
\end{proof}

Of course, when the Nakayama automorphism squares to the identity, then the duality for Tate-Hochschild cohomology is as nice as for the homology. We end this section by recording this in the following corollary.

\begin{corollary}\label{nakayamasquare}
If $\La$ is a Frobenius algebra with a Nakayama automorphism $\nu$ such that $\nu^2 = 1$, then there is an isomorphism
$$\TateHH^n( \La, \La ) \cong  \TateHH^{-(n+1)} ( \La , \La )$$
for all $n \in \mathbb{Z}$.
\end{corollary}

\section{Quantum complete intersections}\label{quantumci}

Quantum complete intersections are noncommutative analogues of truncated polynomial rings, and are obtained by replacing the ordinary commutation relations between the generators by quantum versions. The terminology dates back to work by Avramov, Gasharov and Peeva, and, ultimately, Manin; the notion of \emph{quantum symmetric algebras} was introduced in \cite{Manin}, and that of \emph{quantum regular sequences} in \cite{AvramovGasharovPeeva}.

Fix a field $k$, let $c \ge 1$ be an integer, and let ${\bf{q}} =
(q_{ij})$ be a $c \times c$ commutation matrix with entries in $k$. That is, the diagonal entries $q_{ii}$ are all $1$, and $q_{ij}q_{ji}=1$ for all $i,j$. Furthermore, let ${\bf{a}}_c = (a_1, \dots, a_c)$ be an ordered sequence of $c$ integers with $a_i \ge 2$. The \emph{quantum complete intersection} $A_{\bf{q}}^{{\bf{a}}_c}$ determined by these data is the algebra
$$A_{\bf{q}}^{{\bf{a}}_c} \stackrel{\text{def}}{=} k \langle X_1, \dots, X_c \rangle / (X_i^{a_i}, X_iX_j-q_{ij}X_jX_i),$$
a finite dimensional algebra of dimension $\prod_{i=1}^c a_i$. The image of $X_i$ in this quotient will be denoted by $x_i$. Note that the class of all quantum complete intersections includes the exterior algebras
$$k \langle X_1, \dots, X_c \rangle / (X_i^2, X_iX_j+X_jX_i),$$
as well as finite dimensional commutative complete intersections of the form
$$k[X_1, \dots, X_c]/(X_1^{a_1}, \dots, X_c^{a_c}).$$
These two types of algebras are Frobenius, and the following result shows that this is the case with all quantum complete intersections.

\begin{lemma}\cite[Lemma 3.1]{Bergh}\label{nakayamaQCI}
A quantum complete intersection $A_{\bf{q}}^{{\bf{a}}_c}$ is
Frobenius, with an isomorphism $A_{\bf{q}}^{{\bf{a}}_c}
\xrightarrow{\phi} D(A_{\bf{q}}^{{\bf{a}}_c})$ of left modules and corresponding
Nakayama automorphism $A_{\bf{q}}^{{\bf{a}}_c}
\xrightarrow{\nu} A_{\bf{q}}^{{\bf{a}}_c}$ given by
\begin{eqnarray*}
\phi(1) \left ( \sum_{i_1, \dots, i_c} \alpha_{i_1, \dots, i_c} x_c^{i_c} \cdots x_1^{i_1} \right ) & = & \alpha_{a_1-1, \dots, a_c-1} \\
\nu (x_w) & = & \left ( \prod_{i=1}^c q_{iw}^{a_i-1} \right ) x_w
\end{eqnarray*}
for $1 \le w \le c$.
\end{lemma}

Thus the Tate-Hochschild duality results from the previous section applies to quantum complete intersections, in particular, there are isomorphisms
\begin{eqnarray*}
\TateHH_n( A_{\bf{q}}^{{\bf{a}}_c}, A_{\bf{q}}^{{\bf{a}}_c} ) & \cong & \TateHH_{-(n+1)}( A_{\bf{q}}^{{\bf{a}}_c}, A_{\bf{q}}^{{\bf{a}}_c} ) \\
\TateHH^n( A_{\bf{q}}^{{\bf{a}}_c}, A_{\bf{q}}^{{\bf{a}}_c} ) & \cong & \TateHH^{-(n+1)} ( A_{\bf{q}}^{{\bf{a}}_c} , {_{\nu^2}(A_{\bf{q}}^{{\bf{a}}_c})}_1 )
\end{eqnarray*}
for all $n \in \mathbb{Z}$. As in Corollary \ref{nakayamasquare}, the quantum complete intersections whose Nakayama automorphisms square to the identity satisfies the same nice duality for Tate-Hochschild cohomology as for homology. In particular, this holds for the exterior algebras.

\begin{theorem}\label{exterioralgebras}
If $k$ is a field and $A$ an exterior algebra
$$A = k \langle X_1, \dots, X_c \rangle / (X_i^2, X_iX_j+X_jX_i),$$
then there are isomorphisms
\begin{eqnarray*}
\TateHH_n( A,A ) & \cong & \TateHH_{-(n+1)}( A,A ) \\
\TateHH^n( A,A ) & \cong & \TateHH^{-(n+1)} ( A,A )
\end{eqnarray*}
for all $n \in \mathbb{Z}$.
\end{theorem}

\begin{proof}
Only the cohomology isomorphism needs explanation. By Lemma \ref{nakayamaQCI}, the Nakayama automorphism of $\La$ is the identity when $c$ is odd, and maps a generator $x_i$ to $-x_i$ when $c$ is even. In either case, the automorphism squares to the identity.
\end{proof}

We shall calculate the dimensions of the Tate-Hochschild (co)homology groups of all exterior algebras and certain commutative complete intersections. Moreover, we shall also find lower bounds for the dimensions of the homology groups of a general quantum complete intersection. In order to do this, we need an explicit description of a complete bimodule resolution of these algebras ``near zero". Let therefore $A$ denote a general quantum complete intersection $A_{\bf{q}}^{{\bf{a}}_c}$, and consider the element
$$s = \sum_{\substack{0 \le i_1 <a_1 \\ \vdots \\ 0 \le i_c <a_c}} \left ( \prod_{1 \le u<v \le c} q_{uv}^{-i_v(a_u-i_u-1)} \right ) x_c^{i_c} \cdots x_1^{i_1} \otimes x_c^{a_c-i_c-1} \cdots x_1^{a_1-i_1-1}$$
in the enveloping algebra $\Ae$. Furthermore, let $( \Ae )^c \xrightarrow{f} \Ae$ be the map obtained by multiplying an element of $( \Ae )^c$ by the $c \times 1$ matrix
$$(1 \otimes x_1-x_1 \otimes 1 \cdots 1 \otimes x_c-x_c \otimes 1)^T$$
from the right. We claim that the sequence
$$( \Ae )^c \xrightarrow{f} \Ae \xrightarrow{\cdot s} \Ae$$
of left $\Ae$-homomorphisms is exact. A direct (but tedious) computation shows that, for $1 \le t \le c$, the expression
\begin{eqnarray*}
& & \sum_{\substack{0 \le i_1 <a_1 \\ \vdots \\ 0 \le i_c <a_c}} \left ( \prod_{1 \le u<v \le c} q_{uv}^{-i_v(a_u-i_u-1)} \right ) x_c^{i_c} \cdots x_1^{i_1} \otimes x_c^{a_c-i_c-1} \cdots x_1^{a_1-i_1-1}x_t \\
& - & \sum_{\substack{0 \le i_1 <a_1 \\ \vdots \\ 0 \le i_c <a_c}} \left ( \prod_{1 \le u<v \le c} q_{uv}^{-i_v(a_u-i_u-1)} \right ) x_tx_c^{i_c} \cdots x_1^{i_1} \otimes x_c^{a_c-i_c-1} \cdots x_1^{a_1-i_1-1}
\end{eqnarray*}
is zero in $\Ae$. But this expression is the product $(1 \otimes x_t-x_t \otimes 1)s$, hence the sequence is a complex. Since the cokernel of the left map is $A$, it is enough to show that the dimension of the image of the right map is at least the dimension of $A$, namely $a_1a_2 \cdots a_c$. This is easy: the elements
$$(x_c^{j_c} \cdots x_1^{j_1} \otimes 1)s \hspace{1cm} 0 \le j_1 < a_1, \dots, 0 \le j_c < a_c$$
are linearly independent in $\Ae$.

Consequently the sequence is exact, and may therefore be considered as the part
$$P_1 \xrightarrow{d_1} P_0 \xrightarrow{d_0} P_{-1}$$
of a complete bimodule resolution of $A$. We shall use it to calculate $\TateHH_0(A, {_{\psi}A_1})$ for various twisted bimodules ${_{\psi}A_1}$, with the help of the following lemma.

\begin{lemma}\label{zeromaps}
Let $A=A_{\bf{q}}^{{\bf{a}}_c}$ be a quantum complete intersection, and $A \xrightarrow{\psi} A$ an automorphism given by $x_i \mapsto \alpha_i x_i$, where $\alpha_1, \dots, \alpha_c$ are nonzero scalars. Furthermore, let $e^c_w$ denote the $w$th standard generator in the $c$-fold direct sum ${_{\psi}A_1^c}$, and $\underline{\alpha}$ the element
$$(1+ \alpha_1 + \cdots + \alpha_1^{a_1-1})(1+ \alpha_2 + \cdots + \alpha_2^{a_2-1}) \cdots (1+ \alpha_c + \cdots + \alpha_c^{a_c-1}).$$
Then there is an isomorphism
$$\xymatrix{
{_{\psi}A_1} \otimes_{\Ae} ( \Ae )^c \ar[r]^{1 \otimes f} \ar[d]^{\wr} & {_{\psi}A_1} \otimes_{\Ae} \Ae \ar[r]^{1 \otimes ( \cdot s)} \ar[d]^{\wr} & {_{\psi}A_1} \otimes_{\Ae} \Ae \ar[d]^{\wr} \\
{_{\psi}A_1^c} \ar[r]^{d_1^{\psi}} & {_{\psi}A_1} \ar[r]^{d_0^{\psi}} & {_{\psi}A_1} }$$
of complexes, where the maps $d_i^{\psi}$ are given as follows:
\begin{eqnarray*}
d_1^{\psi} (x_c^{u_c} \cdots x_1^{u_1}e^c_w) & = & \left ( \alpha_w \prod_{i=w}^c q_{wi}^{u_i} - \prod_{j=1}^w q_{jw}^{u_j} \right ) (x_c^{u_c} \cdots x_w^{u_w+1} \cdots x_1^{u_1}) \\
d_0^{\psi} (x_c^{u_c} \cdots x_1^{u_1}) & = & \left \{
\begin{array}{ll}
0 & \text{if } u_i>0 \text{ for one } i \\
\underline{\alpha} x_c^{a_c-1} \cdots x_1^{a_1-1} & \text{if } u_1= \cdots =u_c =0.
\end{array} \right.
\end{eqnarray*}
\end{lemma}

\begin{proof}
Clearly
$$d_1^{\psi} (x_c^{u_c} \cdots x_1^{u_1}e^c_w) = (x_c^{u_c} \cdots x_1^{u_1}) \cdot (1 \otimes x_w - x_w \otimes 1),$$
where the product means right scalar action on ${_{\psi}A_1}$ from $\Ae$. Therefore
\begin{eqnarray*}
d_1^{\psi} (x_c^{u_c} \cdots x_1^{u_1}e^c_w) & = & \psi (x_w) x_c^{u_c} \cdots x_1^{u_1} - x_c^{u_c} \cdots x_1^{u_1}x_w \\
& = & \alpha_w x_w x_c^{u_c} \cdots x_1^{u_1} - x_c^{u_c} \cdots x_1^{u_1}x_w \\
& = & \left ( \alpha_w \prod_{i=w}^c q_{wi}^{u_i} - \prod_{j=1}^w q_{jw}^{u_j} \right ) (x_c^{u_c} \cdots x_w^{u_w+1} \cdots x_1^{u_1}).
\end{eqnarray*}
As for $d_0^{\psi}$, this is just right multiplication with $s$. Since the total weight  of $x_i$ in $s$ is $a_i-1$, it is easy to see that $d_0^{\psi} (x_c^{u_c} \cdots x_1^{u_1})=0$ if $u_i \ge 1$. Thus only $d_0^{\psi}(1)$ remains:
\begin{eqnarray*}
d_0^{\psi}(1) & = & 1 \cdot s \\
& = & \sum_{\substack{0 \le i_1 <a_1 \\ \vdots \\ 0 \le i_c <a_c}} \left ( \prod_{1 \le u<v \le c} q_{uv}^{-i_v(a_u-i_u-1)} \right ) \psi ( x_c^{a_c-i_c-1} \cdots x_1^{a_1-i_1-1} ) x_c^{i_c} \cdots x_1^{i_1} \\
& = & \underline{\alpha} x_c^{a_c-1} \cdots x_1^{a_1-1}.
\end{eqnarray*}
\end{proof}

As a first application of Lemma \ref{zeromaps}, we calculate the Tate-Hochschild (co)homology groups of certain finite dimensional commutative complete intersections.

\begin{theorem}\label{ci}
Let $k$ be a field of characteristic $p$, and $A$ a finite dimensional commutative complete intersection of the form
$$A = k [ X_1, \dots, X_c ] / (X_1^{a}, \dots, x_c^{a}),$$
where $a \ge 2$.
Then
$$\dim \TateHH_n(A,A) = \left \{
\begin{array}{ll}
\binom{c+n-1}{n} a^c & \text{if } p \mid a \\
a^c-1 & \text{if } p \nmid a \text{ and } n=0 \\
\sum_{t=0}^c \binom{c}{t} \binom{n-1}{n-c+t}a^t(a-1)^{c-t} & \text{if } p \nmid a \text{ and } n \ge 1
\end{array}
\right.$$
for $n \ge 0$, and
$$\dim \TateHH_n(A,A) = \dim \TateHH^n(A,A) = \dim \TateHH_{-(n+1)}(A,A) = \dim \TateHH^{-(n+1)}(A,A)$$
for all $n \in \mathbb{Z}$.
\end{theorem}

\begin{proof}
Since $A$ is symmetric, it follows from Lemma \ref{dualityexttor} that the dimension of   $\TateHH^n(A,A)$ equals that of $\TateHH_n(A,A)$ for all $n \in \mathbb{Z}$. Together with Theorem \ref{dualityTHfrobenius}, this gives the three dimension equalities.

To calculate $\TateHH_0(A,A)$, we use Lemma \ref{zeromaps} with $\psi =1$. The map $d^1_1$ is clearly the zero map, hence $\dim \TateHH_0(A,A) = \dim \Ker d^1_0$. Since
$$d_0^1 (x_c^{u_c} \cdots x_1^{u_1}) = \left \{
\begin{array}{ll}
0 & \text{if } u_i>0 \text{ for one } i \\
a^c x_c^{a-1} \cdots x_1^{a-1} & \text{if } u_1= \cdots =u_c =0,
\end{array} \right.$$
we obtain
$$\dim \TateHH_0(A,A) = \left \{
\begin{array}{ll}
a^c -1 & \text{if } p \nmid a \\
a^c & \text{if } p \mid a.
\end{array} \right.$$

\sloppy To calculate $\TateHH_n(A,A)$ for $n \ge 1$, we use the fact that $\dim \TateHH^n(A,A) = \dim \TateHH_n(A,A)$, and that, for such $n$, there is an isomorphism $\TateHH^n(A,A) \cong \HH^n(A,A)$. Moreover, by \cite[Theorem X.7.4]{Maclane}, there is an isomorphism
$$\HH^n(A,A) = \bigoplus_{n_1+ \cdots +n_c =n} \left ( \HH^{n_1} \left ( k[X]/(X^{a}) \right ) \otimes_k \cdots \otimes_k \HH^{n_c} \left ( k[X]/(X^{a}) \right ) \right ),$$
hence
\begin{equation*}\label{formula}
\dim \HH^n(A,A) = \sum_{n_1+ \cdots +n_c =n} \left ( \prod_{i=1}^c \dim \HH^{n_i} \left ( k[X]/(X^{a}) \right ) \right ). \tag{$\dagger$}
\end{equation*}
By \cite[Proposition 2.2]{Holm}, the dimensions of the Hochschild cohomology groups of the truncated polynomial algebra $k[X]/(X^a)$ are given by
$$\dim \HH^n \left ( k[X]/(X^a) \right ) = \left \{
\begin{array}{ll}
a & \text{when } n=0 \\
a & \text{when } n>0 \text{ and } p \mid a \\
a-1 & \text{when } n>0 \text{ and } p \nmid a.
\end{array}
\right.$$
Therefore, when $p \mid a$, then
\begin{eqnarray*}
\dim \HH^n(A,A) &=& \sum_{\substack{n_1+ \cdots +n_c =n \\ n_i \ge 0}} a^c \\
&=& \binom{c+n-1}{n} a^c.
\end{eqnarray*}
When $p \nmid a$, then we have to keep track of how many times $\HH^0 \left ( k[X]/(X^{a}) \right )$ appears in each summand in the formula (\ref{formula}), since now $\dim \HH^0 \left ( k[X]/(X^{a}) \right ) =a$ whereas $\dim \HH^m \left ( k[X]/(X^{a}) \right ) =a-1$ for $m \ge 1$. If exactly $t$ out the numbers $n_1, \dots, n_c$ are zero, then the remaining $c-t$ are nonzero. The number of integer solutions to
$$x_1 + \cdots +x_{c-t} = n, \hspace{5mm} x_i \ge 1$$
is the same as the number of solutions to
$$y_1 + \cdots +y_{c-t} = n-c+t, \hspace{5mm} y_i \ge 0,$$
namely
$$\binom{n-1}{n-c+t}.$$
Therefore, in the formula (\ref{formula}), the total contribution from all the summands in which precisely $t$ out the numbers $n_1, \dots, n_c$ are zero, is
$$\binom{c}{t} \binom{n-1}{n-c+t}a^t(a-1)^{c-t}.$$
Summing up, we see that when $p \nmid a$, then
$$\dim \HH^n(A,A) = \sum_{t=0}^c \binom{c}{t} \binom{n-1}{n-c+t}a^t(a-1)^{c-t}.$$
\end{proof}

\begin{remark}\label{remarkCI}
Theorem \ref{ci} is probably well known to some, at least in terms of the ordinary Hochschild (co)homology. However, we were unable to find a reference. Note that the same method of proof also applies to general finite dimensional commutative complete intersections of the form
$$k [ X_1, \dots, X_c ] / (X_1^{a_1}, \dots, X_c^{a_c}).$$
However, the resulting formulas become much more complicated.
\end{remark}

Next, we calculate the Tate-Hochschild (co)homology groups of all exterior algebras.

\begin{theorem}\label{THexterioralgebras}
Let $k$ be a field of characteristic $p$, and $A$ an exterior algebra
$$A = k \langle X_1, \dots, X_c \rangle / (X_i^2, X_iX_j+X_jX_i).$$
Then
$$\dim \TateHH_n(A,A) = \left \{
\begin{array}{ll}
2^c \binom{c+n-1}{c-1} & \text{if } p =2 \\
2^c-2^{c-1} & \text{if } p \neq 2 \text{ and } n=0 \\
2^{c-1} \binom{c+n-1}{c-1} & \text{if } p \neq 2 \text{ and } n \ge 1
\end{array}
\right.$$
for $n \ge 0$, and
$$\dim \TateHH_n(A,A) = \dim \TateHH^n(A,A) = \dim \TateHH_{-(n+1)}(A,A) = \dim \TateHH^{-(n+1)}(A,A)$$
for all $n \in \mathbb{Z}$.
\end{theorem}

\begin{proof}
For positive $n$, the dimensions of $\TateHH_n(A,A)$ and $\TateHH^n(A,A)$ are given by \cite[Theorem 2 and Theorem 3]{XuHan}, and those results also show equality. In view of Theorem \ref{exterioralgebras}, we therefore only have to calculate $\TateHH_0(A,A)$ and $\TateHH^0(A,A)$.

First we calculate the dimension of $\TateHH_0(A,A)$. In the terminology of Lemma \ref{zeromaps}, we must calculate the homology of the complex
$$A^c \xrightarrow{d^1_1} A \xrightarrow{d^1_0} A,$$
with maps given by
\begin{eqnarray*}
d_1^1 (x_c^{u_c} \cdots x_1^{u_1}e^c_w) & = & \left ( (-1)^{u_{w+1}+ \cdots + u_c}- (-1)^{u_1+ \cdots + u_{w-1}} \right ) (x_c^{u_c} \cdots x_w^{u_w+1} \cdots x_1^{u_1}) \\
d_0^1 (x_c^{u_c} \cdots x_1^{u_1}) & = & \left \{
\begin{array}{ll}
0 & \text{if } u_i>0 \text{ for one } i \\
2^c x_c \cdots x_1 & \text{if } u_1= \cdots =u_c =0.
\end{array} \right.
\end{eqnarray*}
Suppose that $p \neq 2$. Then $x_c^{u_c} \cdots x_1^{u_1} \in \Im d^1_1$ if and only if $u_1 + \cdots + u_c$ is a positive even number, and so
$$\dim \Im d^1_1 = \binom{c}{2} + \binom{c}{4} + \cdots = 2^{c-1}-1.$$
Moreover, the dimension of $\Ker d^1_0$ is $2^c-1$. If $p=2$, then $d^1_1$ and $d^1_0$ are both zero, and consequently
$$\dim \TateHH_0(A,A) = \left \{
\begin{array}{ll}
2^c & \text{if } p =2 \\
2^c-2^{c-1} & \text{if } p \neq 2.
\end{array}
\right.$$

Next, we calculate the dimension of $\TateHH^0(A,A)$. If $c$ is odd or $p=2$, then $A$ is symmetric, and so $\dim \TateHH^0(A,A) = \dim \TateHH_0(A,A)$ in this case. Suppose therefore that $c$ is even and $p \neq 2$. By Lemma \ref{dualityexttor}, the dimension of $\dim \TateHH^0(A,A)$ equals that of $\dim \TateHH_0(A, {_{\nu}A_1})$, and the Nakayama automorphism $\nu$ maps $x_i$ to $-x_i$. Using Lemma \ref{zeromaps}, we must therefore calculate the homology of the complex
$${_{\nu}A^c_1} \xrightarrow{d^{\nu}_1} {_{\nu}A_1} \xrightarrow{d^{\nu}_0} {_{\nu}A_1},$$
with
$$d_1^{\nu} (x_c^{u_c} \cdots x_1^{u_1}e^c_w)  = - \left ( (-1)^{u_{w+1}+ \cdots + u_c}+ (-1)^{u_1+ \cdots + u_{w-1}} \right ) (x_c^{u_c} \cdots x_w^{u_w+1} \cdots x_1^{u_1})$$
and $d_0^{\nu} = 0$. We see that $x_c^{u_c} \cdots x_1^{u_1} \in \Im d^1_1$ if and only if $u_1 + \cdots + u_c$ is a positive odd number, and so
$$\dim \Im d^1_1 = \binom{c}{1} + \binom{c}{3} + \cdots = 2^{c-1}.$$
Consequently, the dimension of $\TateHH^0(A,A)$ equals that of $\TateHH_0(A,A)$ also in this case, namely $2^c-2^{c-1}$.
\end{proof}

The next result establishes lower bounds for the dimensions of the Tate-Hochschild homology groups of an arbitrary quantum complete intersection.

\begin{theorem}\label{lowerbound}
Let $k$ be a field of characteristic $p$, and
$$A_{\bf{q}}^{{\bf{a}}_c} = k \langle X_1, \dots, X_c \rangle / (X_i^{a_i}, X_iX_j-q_{ij}X_jX_i)$$
a quantum complete intersection with $a_i \ge 2$ for all $i$. Furthermore, suppose $p$ divides $d$ of the exponents $a_1, \dots, a_c$. Then
$$\dim \TateHH_n ( A_{\bf{q}}^{{\bf{a}}_c}, A_{\bf{q}}^{{\bf{a}}_c} ) \ge \left \{
\begin{array}{ll}
\sum_{i=1}^ca_i - c & \text{if } n=0,-1 \text{ and } p \nmid a_i \text{ for all } i \\
\sum_{i=1}^ca_i - c+1 & \text{if } n=0,-1 \text{ and } p \mid a_i \text{ for some } i \\
\sum_{i=1}^ca_i - c+ d & \text{if } n \neq 0,-1,
\end{array}
\right.$$
in particular $\TateHH_n ( A_{\bf{q}}^{{\bf{a}}_c}, A_{\bf{q}}^{{\bf{a}}_c} ) \neq 0$ for all $n \in \mathbb{Z}$.
\end{theorem}

\begin{proof}
Denote the algebra by $A$. It follows from \cite[Proposition 4.9]{BerghMadsen} that the dimension of $\TateHH_n(A,A)$ is at least $\sum_{i=1}^ca_i - c+ d$ when $n \ge 1$. Therefore, by Theorem \ref{dualityTHfrobenius}, we only need to establish the bound for the dimension of $\TateHH_0(A,A)$.

As before, we use Lemma \ref{zeromaps}: the space $\TateHH_0(A,A)$ is the homology of the complex
$$A^c \xrightarrow{d^1_1} A \xrightarrow{d^1_0} A,$$
with maps given by
\begin{eqnarray*}
d_1^1 (x_c^{u_c} \cdots x_1^{u_1}e^c_w) & = & \left ( \prod_{i=w}^c q_{wi}^{u_i} - \prod_{j=1}^w q_{jw}^{u_j} \right ) (x_c^{u_c} \cdots x_w^{u_w+1} \cdots x_1^{u_1}) \\
d_0^1 (x_c^{u_c} \cdots x_1^{u_1}) & = & \left \{
\begin{array}{ll}
0 & \text{if } u_i>0 \text{ for one } i \\
\left ( \prod_{i=1}^c a_i \right ) x_c^{a_c-1} \cdots x_1^{a_1-1} & \text{if } u_1= \cdots =u_c =0.
\end{array} \right.
\end{eqnarray*}
For any $1 \le w \le c$ and $1 \le u \le a_w-1$, the element $x_w^u$ in $A$ is not contained in the image of $d^1_1$, because $d^1_1(x_w^{u-1}e^c_w)=0$. Also, the identity in $A$ is not contained in this image, hence
$$\dim \Im d^1_1 \le \dim A - \sum_{i=1}^c (a_i-1) -1.$$
As for the dimension of $\Ker d^1_0$, this is $\dim A$ when $p$ divides one of the exponents $a_1, \dots, a_c$, and $\dim A-1$ if not. The lower bound for the dimension of $\TateHH_0(A,A)$ follows immediately from this.
\end{proof}

Of course, for some quantum complete intersections, the difference between the actual dimensions of the Tate-Hochschild homology groups and the lower bound given in Theorem \ref{lowerbound} can be arbitrarily large. For example, suppose that $A$ is either a finite dimensional commutative complete intersection of the form
$$k [ X_1, \dots, X_c ] / (X_1^{a}, \dots, x_c^{a}),$$
or an exterior algebra
$$k \langle X_1, \dots, X_c \rangle / (X_i^2, X_iX_j+X_jX_i).$$
Then Theorem \ref{ci} and Theorem \ref{THexterioralgebras} show that
$$\lim_{n \to \pm \infty} \dim \TateHH_n(A,A) = \infty,$$
in fact, when $n \ge 1$, then $\dim \TateHH_n(A,A)$ is given by a polynomial of degree $c-1$.

However, as the following result shows, there are quantum complete intersections where the lower bound given in Theorem \ref{lowerbound} is the actual dimension of the Tate-Hochschild homology groups.

\begin{theorem}\label{homologycodim2}
Let $k$ be a field of characteristic $p$, and
$$A = k \langle X,Y \rangle / (X^a, XY-qYX, Y^b)$$
a quantum complete intersection with $a,b\ge 2$ and $q$ not a root of unity in $k$. Then
$$\dim \TateHH_n(A,A) = \left \{
\begin{array}{ll}
a+b-2 & \text{if } n=0,-1 \text{ and } p \nmid a,b \\
a+b-1 & \text{if } n=0,-1 \text{ and } p \mid a \text { or } p \mid b \\
a+b-2 & \text{if } n \neq 0,-1 \text{ and } p \nmid a,b \\
a+b-1 & \text{if } n \neq 0,-1 \text{ and either } p \mid a \text { or } p \mid b \\
a+b & \text{if } n \neq 0,-1 \text{ and } p \mid a,b.
\end{array}
\right.$$
\end{theorem}

\begin{proof}
The dimensions of $\TateHH_n(A,A)$ for $n \ge 1$ follow from \cite[Theorem 3.1]{BerghErdmann}, and so by Theorem \ref{dualityTHfrobenius}, we only need to calculate the dimension of $\TateHH_0(A,A)$. By Lemma \ref{zeromaps}, this homology group is the homology of the complex
$$A^2 \xrightarrow{d_1^1} A \xrightarrow{d_0^1} A,$$
with maps given by
\begin{eqnarray*}
d_1^1 (y^u x^ve^2_1) & = & (q^u-1)y^ux^{v+1} \\
d_1^1 (y^u x^ve^2_2) & = & (1-q^v)y^{u+1}x^v \\
d_0^1 (y^u x^v) & = & \left \{
\begin{array}{ll}
0 & \text{if } u>0 \text{ or } v>0 \\
ab y^{b-1}x^{a-1} & \text{if } u=v=0.
\end{array} \right.
\end{eqnarray*}
It is easy to see that an element $y^u x^v \in A$ belongs to $\Im d^1_1$ if and only if  both $u$ and $v$ are positive: for then $d_1^1 (y^{u-1} x^ve^2_2) = (1-q^v)y^ux^v$, and  $1-q^v$ is nonzero since $q$ is not a root of unity. This shows that $\dim \Im d^1_1 = (a-1)(b-1) = ab-a-b+1$. The dimension of $\Ker d^1_0$ is $ab-1$ if $p$ does not divide any of $a,b$, and $ab$ if not. The dimension of $\TateHH_0(A,A)$ follows from this.
\end{proof}

When it comes to cohomology, the situation is totally different than for homology. On the one hand, if $A$ is either a finite dimensional commutative complete intersection, or an exterior algebra, then from Theorem \ref{ci} and Theorem \ref{THexterioralgebras} we see that $\TateHH^n(A,A)$ is nonzero for all $n \in \mathbb{Z}$. In fact, just as for homology, when $n \ge 1$, then $\dim \TateHH^n(A,A)$ is given by a polynomial of degree $c-1$ (where $c$ is the number of defining generators for $A$). On the other hand, if $A$ is as in Theorem \ref{homologycodim2}, then it follows from \cite[Theorem 3.2]{BerghErdmann} that $\TateHH^n(A,A) =0$ when $n \ge 3$. Consequently, there is no cohomological counterpart to Theorem \ref{lowerbound}: there is no universal lower bound for the dimensions of the Tate-Hochschild cohomology groups of quantum complete intersections.

Our final main result in this paper is the cohomological version of Theorem \ref{homologycodim2}; we shall determine the dimensions of all the Tate-Hochschild cohomology groups of the quantum complete intersection
$$A = k \langle X,Y \rangle / (X^a, XY-qYX, Y^b)$$
when $q$ is not a root of unity in $k$. To do this, we first calculate the dimension of the (ordinary) Hochschild homology group $\HH_n(A, {_{\nu^{-1}}A_1})$ for $n \ge 1$.
By Lemma \ref{nakayamaQCI}, the automorphism $\nu^{-1}$ is given by $x \mapsto q^{b-1}x, y \mapsto q^{1-a}y$.

For parameters $t,i,u,v$, all non-negative integers, define the following eight scalars:
\begin{align*}
K_1(t,i,u,v) & =  q^{a+b-ab-1} \sum_{j=0}^{b-1} q^{j \left ( a+ \frac{ai}{2}+v-1 \right ) } & i \text{ even }, i \le 2t \\
K_2(t,i,u,v) & =  \sum_{j=0}^{a-1} q^{j \left ( bt+b- \frac{bi}{2} +u-1 \right ) } & i \text{ even }, i \le 2t \\
K_3(t,i,u,v) & =  q^{\frac{ai-a+2+2v}{2}} - q^{1-a} & i \text{ odd }, i \le 2t-1 \\
K_4(t,i,u,v) & =  q^{\frac{2bt-bi+b+2u}{2}} -1 & i \text{ odd }, i \le 2t-1 \\
K_5(t,i,u,v) & =  q^{1-a} - q^{\frac{ai+2v}{2}}  & i \text{ even }, i \le 2t \\
K_6(t,i,u,v) & =  \sum_{j=0}^{a-1} q^{j \left ( bt+b- \frac{bi}{2} +u \right ) } & i \text{ even }, i \le 2t \\
K_7(t,i,u,v) & =  q^{a+b-ab-1} \sum_{j=0}^{b-1} q^{j \left ( a+ \frac{a(i-1)}{2}+v \right ) } & i \text{ odd }, i \le 2t+1 \\
K_8(t,i,u,v) & =  q^{\frac{2bt-bi+3b+2u-2}{2}} - 1  & i \text{ odd }, i \le 2t+1.
\end{align*}
Note that since $q$ is not a root of unity and $a,b \ge 2$, all these scalars are nonzero in $k$. Next, for each integer $n \ge 0$, denote by $\oplus_{i=0}^n A e^n_i$ the vector space consisting of $n+1$ copies of $A$. Finally, for each $n \ge 1$, define a map
$$\oplus_{i=0}^n Ae^n_i \xrightarrow{\delta_n} \oplus_{i=0}^{n-1} Ae^{n-1}_i$$
by
$$
\begin{array}{l}
\delta_{2t} \colon  y^ux^v e^{2t}_i \mapsto \vspace{2mm} \\
\left \{
\begin{array}{ll}
K_1(t,i,u,v)y^{u+b-1}x^ve^{2t-1}_i +
K_2(t,i,u,v)y^ux^{v+a-1}e^{2t-1}_{i-1}, & \text{ for $i$
even}
\\
\\
K_3(t,i,u,v)y^{u+1}x^ve^{2t-1}_i + K_4(t,i,u,v)y^ux^{v+1}e^{2t-1}_{i-1}, & \text{
  for $i$ odd}
\end{array} \right. \\
\\
\delta_{2t+1} \colon y^ux^v e^{2t+1}_i \mapsto \vspace{2mm} \\
\left \{
\begin{array}{ll}
K_5(t,i,u,v)y^{u+1} x^ v e^{2t}_i +
K_6(t,i,u,v) y^u x^{v+a-1} e^{2t}_{i-1}, & \text{ for $i$
even}
\\
\\
K_7(t,i,u,v) y^{u+b-1} x^v e^{2t}_i + K_8(t,i,u,v)y^u x^{v+1} e^{2t}_{i-1}, &
\text{ for $i$ odd,}
\end{array} \right.
\end{array}
$$
where we use the convention $e^n_{-1} = e^n_{n+1} =0$. With this notation, it follows from \cite[page 510-511]{BerghErdmann} that $\HH_n(A, {_{\nu^{-1}}A_1})$ is the homology of the complex
$$\cdots \to \oplus_{i=0}^{n+1} Ae^{n+1}_i \xrightarrow{\delta_{n+1}}
\oplus_{i=0}^n Ae^n_i \xrightarrow{\delta_n} \oplus_{i=0}^{n-1} Ae^{n-1}_i \to \cdots$$
of $k$-vector spaces.

\begin{proposition}\label{negativecohomologycodim2}
Let $k$ be a field and
$$A = k \langle X,Y \rangle / (X^a, XY-qYX, Y^b)$$
a quantum complete intersection with $a,b\ge 2$ and $q$ not a root of unity in $k$. Then $\HH_n(A, {_{\nu^{-1}}A_1})=0$ for $n \ge 1$.
\end{proposition}

\begin{proof}
We first compute the kernel of $\delta_{2t}$ for $t \ge 1$. If $i$ is even, then
$$\delta_{2t}(y^ux^ve^{2t}_i) =0 \Leftrightarrow \left \{
\begin{array}{l}
u \ge 1, v \ge 1, i \in \{ 0,2, \dots, 2t \}, \text{ or} \\
u \ge 1, v=0, i=0, \text{ or} \\
u=0, v \ge 1, i=2t.
\end{array}
\right.$$
There are $(b-1)(a-1)(t+1)+(b-1)+(a-1)$ such vectors. If $i$ is odd, then
$$\delta_{2t}(y^ux^ve^{2t}_i) =0 \Leftrightarrow u=b-1, v=a-1, i \in \{ 1,3, \dots, 2t-1 \},$$
and there are $t$ such vectors. Finally, the nontrivial linear combinations in $\Ker \delta_{2t}$ are
\begin{align*}
& x^ve^{2t}_i + C_1(t,i,u,v)y^{b-1}x^{v-1}e^{2t}_{i+1} & v \ge 1, i \in \{ 0, 2, \dots, 2t-2 \} \\
& y^ue^{2t}_i + C_2(t,i,u,v)y^{u-1}x^{a-1}e^{2t}_{i-1} & u \ge 1, i \in \{ 2,4, \dots, 2t \},
\end{align*}
where $C_1(t,i,u,v)$ and $C_2(t,i,u,v)$ are suitable nonzero scalars in $k$. There are $(a+b-2)t$ such linear combinations in total. Summing up, we see that the dimension of $\Ker \delta_{2t}$ is $abt+ab-1$.

Next, we compute the kernel of $\delta_{2t+1}$ for $t \ge 0$. If $i$ is even, then
$$\delta_{2t+1}(y^ux^ve^{2t+1}_i) =0 \Leftrightarrow \left \{
\begin{array}{l}
u=b-1, v \ge 1, i \in \{ 0,2, \dots, 2t \}, \text{ or} \\
u=b-1, v=0, i=0.
\end{array}
\right.$$
There are $(a-1)(t+1)+1$ such vectors. If $i$ is odd, then
$$\delta_{2t+1}(y^ux^ve^{2t+1}_i) =0 \Leftrightarrow \left \{
\begin{array}{l}
u \ge 1, v=a-1, i \in \{ 1,3, \dots, 2t+1 \}, \text{ or} \\
u=0, v=a-1, i=2t+1,
\end{array}
\right.$$
and there are $(b-1)(t+1)+1$ such vectors. Finally, the nontrivial linear combinations in $\Ker \delta_{2t+1}$ are
\begin{align*}
& y^ux^ve^{2t+1}_i + C_3(t,i,u,v)y^{u+1}x^{v-1}e^{2t+1}_{i+1} & u \le b-2, v \ge 1, i \in \{ 0, 2, \dots, 2t \} \\
& y^{b-1}e^{2t+1}_i + C_4(t,i,u,v)x^{a-1}e^{2t+1}_{i-1} & i \in \{ 2,4, \dots, 2t \},
\end{align*}
for suitable nonzero scalars $C_3(t,i,u,v)$ and $C_4(t,i,u,v)$ in $k$. There are $(b-1)(a-1)(t+1)+t$ such linear combinations. Consequently, the total dimension of $\Ker \delta_{2t+1}$ is $abt+ab+1$.

We have shown that when $n \ge 1$, then
$$\dim \Ker \delta_n = \left \{
\begin{array}{ll}
ab \frac{n+2}{2} -1 & \text{for $n$ even} \\
ab \frac{n+1}{2} +1 & \text{for $n$ odd}.
\end{array}
\right.$$
The exact sequence
$$0 \to \Ker \delta_n \to \oplus_{i=0}^n Ae^n_i \xrightarrow{\delta_n} \Im \delta_n \to 0$$
gives $\dim \Im \delta_n = (n+1)ab- \dim \Ker \delta_n$, and so
$$\dim \Im \delta_{n+1} = \left \{
\begin{array}{ll}
ab \frac{n+2}{2} -1 & \text{for $n$ even} \\
ab \frac{n+1}{2} +1 & \text{for $n$ odd}.
\end{array}
\right.$$
This shows that $\HH_n(A, {_{\nu^{-1}}A_1})=0$ for $n \ge 1$.
\end{proof}

Using Proposition \ref{negativecohomologycodim2}, we can now compute all the Tate-Hochschild cohomology groups of $A$. Note that the characteristic of the ground field does not matter, contrary to the homology case in Theorem \ref{homologycodim2}.

\begin{theorem}\label{cohomologycodim2}
Let $k$ be a field and
$$A = k \langle X,Y \rangle / (X^a, XY-qYX, Y^b)$$
a quantum complete intersection with $a,b\ge 2$ and $q$ not a root of unity in $k$. Then
$$\dim \TateHH^n(A,A) = \left \{
\begin{array}{ll}
1 & \text{if } n=0 \\
2 & \text{if } n=1 \\
1 & \text{if } n=2 \\
0 & \text{if } n \neq 0,1,2. \\
\end{array}
\right.$$
\end{theorem}

\begin{proof}
For $n \ge 1$, the dimensions follow from \cite[Theorem 3.2]{BerghErdmann}. Moreover, by Theorem \ref{dualityTHfrobenius} and Lemma \ref{dualityexttor} there are equalities
\begin{eqnarray*}
\dim \TateHH^n(A,A) &=& \dim \TateHH^{-(n+1)}(A, {_{\nu^2}A_1}) \\
&=& \dim \TateHH_{-(n+1)}(A, D({_{\nu^2}A_1})) \\
&=& \dim \TateHH_{-(n+1)}(A, {_{\nu^{-1}}A_1}) \\
\end{eqnarray*}
for all $n \in \mathbb{Z}$. It follows from Proposition \ref{negativecohomologycodim2} that $\TateHH_n(A, {_{\nu^{-1}}A_1}) =0$ for $n \ge 1$, hence $\TateHH^n(A,A)=0$ for $n \le -2$. What remains is therefore to compute $\TateHH^0(A,A)$ and $\TateHH^{-1}(A,A)$.

Since $\dim \TateHH^0(A,A) = \dim \TateHH_0(A, {_{\nu}A_1})$ by Lemma \ref{dualityexttor}, we use Lemma \ref{zeromaps}. Namely, the space $\TateHH_0(A, {_{\nu}A_1})$ is the homology of the complex
$${_{\nu}A^2_1} \xrightarrow{d^{\nu}_1} {_{\nu}A_1} \xrightarrow{d^{\nu}_0} {_{\nu}A_1},$$
with maps given by
\begin{eqnarray*}
d^{\nu}_1(y^ux^ve^2_1) &=& (q^{u+1-b}-1)y^ux^{v+1} \\
d^{\nu}_1(y^ux^ve^2_2) &=& (q^{a-1}-q^{v})y^{u+1}x^v \\
d^{\nu}_0(y^ux^v) &=& \left \{
\begin{array}{ll}
0 & \text{if $u \ge 1$ or $v \ge 1$} \\
\frac{q^{a-ba}-1}{q^{1-b}-1} \frac{q^{ba-b}-1}{q^{a-1}-1} y^{b-1}x^{a-1} & \text{if $u=v=0$}.
\end{array}
\right.
\end{eqnarray*}
We see that
$$y^ux^v \in \Im d^{\nu}_1 \Leftrightarrow (u,v) \notin \{ (0,0),(b-1,a-1) \},$$
hence $\dim \Im d^{\nu}_1 = ab-2$. Since $\dim \Ker d^{\nu}_0 = ab-1$, it follows that the dimension of $\TateHH^0(A,A)$ is $1$.

\sloppy Finally, we compute $\TateHH^{-1}(A,A)$. From the beginning of the proof we know that $\dim \TateHH^{-1}(A,A) = \dim \TateHH_0(A, {_{\nu^{-1}}A_1})$, so once again we use Lemma \ref{zeromaps}. The space $\TateHH_0(A, {_{\nu^{-1}}A_1})$ is the homology of the complex
$${_{\nu^{-1}}A^2_1} \xrightarrow{d^{\nu^{-1}}_1} {_{\nu^{-1}}A_1} \xrightarrow{d^{\nu^{-1}}_0} {_{\nu^{-1}}A_1},$$
with maps given by
\begin{eqnarray*}
d^{\nu^{-1}}_1(y^ux^ve^2_1) &=& (q^{u+b-1}-1)y^ux^{v+1} \\
d^{\nu^{-1}}_1(y^ux^ve^2_2) &=& (q^{1-a}-q^{v})y^{u+1}x^v \\
d^{\nu^{-1}}_0(y^ux^v) &=& \left \{
\begin{array}{ll}
0 & \text{if $u \ge 1$ or $v \ge 1$} \\
\frac{q^{b-ba}-1}{q^{1-a}-1} \frac{q^{ba-b}-1}{q^{b-1}-1} y^{b-1}x^{a-1} & \text{if $u=v=0$}.
\end{array}
\right.
\end{eqnarray*}
Here we see that
$$y^ux^v \in \Im d^{\nu^{-1}}_1 \Leftrightarrow (u,v) \neq (0,0),$$
hence $\dim \Im d^{\nu^{-1}}_1 = ab-1$. Since $\dim \Ker d^{\nu^{-1}}_0 = ab-1$, it follows that $\TateHH^{-1}(A,A)=0$.
\end{proof}


\begin{thebibliography}{AvM}
\bibitem[AnF]{AndersonFuller}F.\ Anderson, K.\ Fuller, \emph{Rings and categories of modules}, second edition, Graduate Texts in Mathematics, 13, Springer-Verlag, New York, 1992, x+376 pp.
\bibitem[AGP]{AvramovGasharovPeeva}L.\ Avramov, V.\ Gasharov, I.\ Peeva,
  \emph{Complete intersection dimension}, Publ.\ Math.\ I.H.E.S.\ 86
  (1997), 67-114.
\bibitem[AvM]{AvramovMartsinkovsky}L.\ Avramov, A.\ Martsinkovsky, \emph{Absolute, relative, and Tate cohomology of modules of finite Gorenstein dimension}, Proc.\ London Math.\ Soc.\ (3) 85 (2002), no.\ 2, 393-440.
\bibitem[Ber]{Bergh}P.A.\ Bergh, \emph{$\Ext$-symmetry over quantum complete intersections}, Arch.\ Math.\ 92 (2009), no.\ 6, 566-573.
\bibitem[BeE]{BerghErdmann}P.A.\ Bergh, K.\ Erdmann, \emph{Homology and cohomology of
quantum complete intersections}, Algebra Number Theory 2 (2008), no.\ 5, 501-522.
\bibitem[BeM]{BerghMadsen}P.A.\ Bergh, D.\ Madsen, \emph{Hochschild homology and split pairs}, Bull.\ Sci.\ Math.\ 134 (2010), no.\ 7, 665-676.
\bibitem[CaE]{CartanEilenberg}H.\ Cartan, S.\ Eilenberg, \emph{Homological
    Algebra}, Princeton University Press, Princeton, N.\ J., 1956, xv+390 pp.
\bibitem[ChJ]{ChristensenJorgensen}L.\ W.\ Chistensen, D.\ A.\ Jorgensen, \emph{Tate (co)homology via pinched tensor and pinched Hom of complexes}, in preparation.
\bibitem[EuS]{EuSchedler}C.-H.\ Eu, T.\ Schedler, \emph{Calabi-Yau Frobenius algebras} J.\ Algebra 321 (2009), no.\ 3, 774-815.
\bibitem[Ho1]{Hochschild1}G.\ Hochschild, \emph{On the cohomology groups of an associative algebra}, Ann.\ of Math.\ (2) 46 (1945), 58-67.
\bibitem[Ho2]{Hochschild2}G.\ Hochschild, \emph{On the cohomology theory for associative algebras}, Ann.\ of Math.\ (2) 47 (1946), 568-579.
\bibitem[Hol]{Holm}T.\ Holm, \emph{Hochschild cohomology rings of algebras $k[X]/(f)$}, Beitr{\"a}ge Algebra Geom.\ 41 (2000), no.\ 1, 291?301.
\bibitem[Mac]{Maclane}S.\ Mac Lane, \emph{Homology}, Die Grundlehren der mathematischen Wissenschaften, Bd.\ 114 Academic Press, Inc., Publishers, New York; Springer-Verlag, Berlin-G{\"o}ttingen-Heidelberg 1963 x+422 pp.
\bibitem[Man]{Manin}I.\ Manin, \emph{Some remarks on Koszul algebras
    and quantum groups}, Ann.\ Inst.\ Fourier (Grenoble) 37 (1987),
  191-205.
\bibitem[Tat]{Tate}J.\ Tate, \emph{The higher dimensional cohomology groups of class field theory},
    Ann.\ of Math.\ (2)  56,  (1952) 294-297.
\bibitem[XuH]{XuHan}Y.\ Xu, Y.\ Han, \emph{Hochschild (Co)homology of Exterior Algebras}, Comm.\ Algebra 35 (2007), 115-131.
\end{thebibliography}
\end{document}